\documentclass[12pt,
a4paper]{article}

\usepackage{hyperref}
\usepackage{amsfonts}
\usepackage{mathrsfs}
\usepackage{comment}
\usepackage{amsmath}
\usepackage{amssymb}
\usepackage{amsthm}
\usepackage{amscd}
\usepackage{graphicx}
\usepackage{indentfirst}
\usepackage[all]{xy}
\usepackage{titlesec}
\usepackage[top=22mm, bottom=22mm, left=22mm, right=22mm]{geometry}
\usepackage{enumerate}
\usepackage{bm}
\usepackage{enumitem}
\usepackage{color}
\usepackage{dsfont}
\newtheorem{theorem}{Theorem}[]
\newtheorem{lemma}[theorem]{Lemma}

\newtheorem{conjecture}[theorem]{Conjecture}
\newtheorem{definition}[theorem]{Definition}
\newtheorem{construction}[theorem]{Construction}
\newtheorem{remark}[theorem]{Remark}
\newtheorem{claim}[theorem]{Claim}

\newtheorem{question}[theorem]{Question}

\newtheorem{observation}[theorem]{Observation}
\newtheorem{fact}[theorem]{Fact}
\newcommand{\ma}{\mathcal}

\newcommand{\mr}{\mathscr}
\newcommand{\s}{\subseteq}

\newcommand{\fr}{\frac}
\newcommand{\lc}{\lceil}
\newcommand{\rc}{\rceil}

\newcommand{\e}{{\rm{ex}}}

\begin{document}
\title{Degenerate Tur\'an densities of sparse hypergraphs}
\author{Chong Shangguan and Itzhak Tamo\footnote{The authors are with the Department of Electrical Engineering - Systems, Tel Aviv University, Tel Aviv 6997801, Israel. Email: theoreming@163.com (C. Shangguan), zactamo@gmail.com (I. Tamo)}
}

\date{}
\maketitle

\begin{abstract}
\noindent For fixed integers $r>k\ge 2,e\ge 3$, let $f_r(n,er-(e-1)k,e)$ be the maximum number of edges in an $r$-uniform hypergraph in which the union of any $e$ distinct edges contains at least $er-(e-1)k+1$ vertices.
A classical result of Brown, Erd\H{o}s and S\'os in 1973 showed that $f_r(n,er-(e-1)k,e)=\Theta(n^k).$
The degenerate Tur\'an density is defined to be the limit (if it exists) $$\pi(r,k,e):=\lim_{n\rightarrow\infty}\frac{f_r(n,er-(e-1)k,e)}{n^k}.$$
Extending a recent result of Glock for the special case of $r=3,k=2,e=3$, we show that
$$\pi(r,2,3):=\lim_{n\rightarrow\infty}\frac{f_r(n,3r-4,3)}{n^2}=\frac{1}{r^2-r-1}$$ for arbitrary fixed $r\ge 4$.
For the more general cases $r>k\ge 3$, we manage to show
$$\frac{1}{r^k-r}\le\liminf_{n\rightarrow\infty}\frac{f_r(n,3r-2k,3)}{n^k}\le\limsup_{n\rightarrow\infty}\frac{f_r(n,3r-2k,3)}{n^k}\le \frac{1}{k!\binom{r}{k}-\frac{k!}{2}},$$
where the gap between the upper and lower bounds are small for $r\gg k$.

The main difficulties in proving these results are the constructions establishing the lower bounds. The first construction is recursive and purely combinatorial, and is based on a (carefully designed) approximate induced decomposition of the complete graph, whereas the second construction
is algebraic, and is proved by a newly defined matrix property which we call {\it strongly $3$-perfect hashing}.
\end{abstract}

\section{Introduction}\label{intro}

\noindent Tur\'an-type problems have been playing a central role in the field of extremal graph theory since Tur\'an \cite{turan} determined in 1941 the Tur\'an number of complete graphs.
In this work we focus on a classical hypergraph Tur\'an-type problem introduced by Brown, Erd\H{o}s and S\'os \cite{BES71} in 1973.

For an integer $r\ge 2$, an $r$-uniform hypergraph $\ma{H}$ (or $r$-graph, for short) on the vertex set $V(\ma{H})$, is a family of $r$-element subsets of $V(\ma{H})$,  called the {\it edges} of $\ma{H}$.
An $r$-graph is said to contain a {\it copy} of $\ma{H}$ if it contains $\ma{H}$ as a subhypergraph.
Furthermore, given a family $\mr{H}$ of $r$-graphs, an $r$-graph is said to be {\it $\mr{H}$-free} if it contains no copy of any member of $\mr{H}$.
The {\it Tur\'an number} $\e_r(n,\mr{H})$, is the maximum number of edges in an $\mr{H}$-free $r$-graph on $n$ vertices.
It can be easily shown
that the sequence
$\{\big(\binom{n}{r}\big)^{-1}\cdot\e_r(n,\mr{H})\}_{n=r}^{\infty}$ is bounded and non-increasing, and therefore converges \cite{kns}.
Hence, the {\it Tur\'an density} $\pi(\mr{H})$ of $\mr{H}$ is defined to be
$$\pi(\mr{H}):=\lim_{n\rightarrow\infty}\frac{\e_r(n,\mr{H})}{\binom{n}{r}}.$$
If $\pi(\mr{H})=0$ then $\mr{H}$ is called {\it degenerate}.
It is well-known  (see, e.g. \cite{Erdos-Turan,Keevashsurvey,Kovari-turanbound}) that $\mr{H}$ is degenerate if and only if it contains an $r$-partite $r$-graph,
where an $r$-graph is called {\it $r$-partite} if its vertex set admits a partition into $r$ disjoint parts $V_1,\ldots,V_r$,
such that every edge of it contains exactly one vertex from each vertex part $V_i$.
If $\mr{H}$ is degenerate and there exists a real number $\alpha\in(0,r)$ such that $\e_r(n,\mr{H})=\Theta(n^{\alpha})$, then the {\it degenerate  Tur\'an density} $\pi_d(\mr{H})$ of $\mr{H}$ is defined to be the limit (if it exists)
$$\pi_d(\mr{H}):=\lim_{n\rightarrow\infty}\frac{\e_r(n,\mr{H})}{n^{\alpha}},$$
where $\alpha$ is  called the {\it Tur\'an exponent\footnote{We remark that in a ``reverse'' direction, a recent breakthrough of Bukh and Conlon \cite{Bukhturanexponent} proved that
for any rational number $\alpha\in[1,2]$, there exists a finite family of $2$-graphs $\mr{H}$ such that $\e_2(n,\mr{H})=\Theta(n^{\alpha})$,
which resolves a well-known conjecture of Erd\H{o}s and Simonovits \cite{ErdosTuranExponents1}.
Similar results on $r$-graphs for all $r\ge 3$ were recently obtained by Fitch \cite{fitch2016rational}.
However, another conjecture of Erd\H{o}s and Simonovits \cite{ErdosTuranExponents2}, also known as the rational exponents conjecture,
which claims that in the statement above, it suffices to pick a simple graph rather than a finite family, is still widely open.}} of $\mr{H}$.
For example,  it is known (see, e.g. \cite{furedi2013history}) that $\pi_d(C_4)=\lim_{n\rightarrow\infty}\frac{\e_2(n,C_4)}{n^{3/2}}=\frac{1}{2}$, where $C_4$ is the cycle of length 4.

For a positive integer $n$ let $[n]:=\{1,\ldots,n\}$, and for any $X\s[n]$ let $\binom{X}{r}$ be the family  of $\binom{|X|}{r}$ distinct $r$-subsets of $X$.
For fixed integers $r\ge 2,e\ge 2,v\ge r+1$, let $\mr{G}_r(v,e)$ be the family of all $r$-graphs formed by $e$ edges and at most $v$ vertices; that is, $$\mr{G}_r(v,e)=\Big\{\ma{H}\s\binom{[v]}{r}:|\ma{H}|=e,|V(\ma{H})|\le v\Big\}.$$
Thus an $r$-graph is $\mr{G}_r(v,e)$-free if and only if the union of any $e$ distinct edges contains at least $v+1$ vertices.
Since such  $r$-graphs do not contain many edges (see  \eqref{BESbound} below),  they are also termed {\it sparse hypergraphs} \cite{sparse}.
Following previous papers on this topic (see, e.g. \cite{AlonShapira}) we will use the notation
 $$f_r(n,v,e):=\e_r(n,\mr{G}_r(v,e))$$
to denote the maximum number of edges in a $\mr{G}_r(v,e)$-free $r$-graph.

In 1973, Brown, Erd\H{o}s and S\'os \cite{BES71} initiated the study of the function $f_r(n,v,e)$, which has attracted considerable attention throughout the years.
More concretely, they showed that
\begin{equation}\label{BESbound}
  \begin{aligned}
   \Omega(n^{\fr{er-v}{e-1}})=f_r(n,v,e)=O(n^{\lc\fr{er-v}{e-1}\rc}).
  \end{aligned}
\end{equation}
The lower bound was proved by a standard probabilistic argument (now known as the alteration method, see, e.g. Chapter 3 of \cite{alon2016probabilistic}), and the upper bound follows from a double counting argument, which uses the simple fact that in a $\mr{G}_r(v,e)$-free $r$-graph, any set of $\lc\fr{er-v}{e-1}\rc$ vertices can be contained in at most $e-1$ distinct edges.
Improvements on \eqref{BESbound} for less general parameters were obtained in a series  of  works, see, e.g. \cite{AlonShapira,BrownTriangulated,BES71,Erdos1964,EPR,Erdosr=2,ge2017sparse,glock2018triple,Rodl,Rodlpacking,Ruzsa-Szemeredi,other2,other1,shangguan2019}.

In this paper we are interested in the special case where $k:=\frac{er-v}{e-1}$ is an integer greater than one. In such a case the order of  $f_r(n,v,e)$ is determined by \eqref{BESbound}, i.e.,
\begin{equation}\label{order}
  \begin{aligned}
    f_r(n,er-(e-1)k,e)=\Theta(n^k),
  \end{aligned}
\end{equation}
where  $v=er-(e-1)k$ and $2\le k\le r-1$. Thus for fixed integers $e\ge 2, r>k\ge 2$, it is natural to ask whether the limit
$$\pi_d(r,k,e):=\pi_d\big(\mr{G}_r(er-(e-1)k,e)\big)=\lim_{n\rightarrow\infty}\frac{f_r(n,er-(e-1)k,e)}{n^k}$$ exists,
where we call $\pi_d(r,k,e)$ {\it the degenerate Tur\'an density of sparse hypergraphs}.

For $e=2$ this question is already resolved, since  an $r$-graph is $\mr{G}_r(2r-k,2)$-free if and only if any pair of its edges share at most $k-1$ vertices,
therefore $f_r(n,2r-k,2)$ is equal to the maximum size of an $(n,r,k)$-packing, where an {\it $(n,r,k)$-packing} is a family of $r$-subsets of $[n]$ such that
any $k$-subset of $[n]$ is contained in at most one member of this family.
Clearly, the largest size of an $(n,r,k)$-packing cannot exceed $\binom{n}{k}/\binom{r}{k}$.
Moreover, it was shown by R{\"o}dl \cite{Rodlpacking} (see \cite{glock2016existence,keevash2014existence} for the current state-of-the-art) that for fixed $r,k$ and sufficiently large $n$,
this bound is essentially tight, up to a $1-o(1)$ factor (where $o(1)\rightarrow 0$ as $n\rightarrow\infty$). This implies that
$$\pi_d(r,k,2)=\lim_{n\rightarrow\infty}\frac{(1-o(1))\binom{n}{k}/\binom{r}{k}}{n^k}=\frac{1}{r\cdots(r-k+1)}.$$

For $e\geq 3$ not much is known, and only recently the existence of $\pi_d(3,2,3)$ was resolved.
Brown, Erd\H{o}s and S\'os \cite{BES71}  posed the following conjecture (see also \cite{BrownTriangulated}).
\begin{conjecture}[Brown, Erd\H{o}s and S\'os \cite{BES71}]\label{besconjecture}
  The degenerate Tur\'an density
  $$\pi_d(3,2,e)=\lim_{n\rightarrow\infty}\fr{f_3(n,e+2,e)}{n^2}$$ exists for every fixed $e\ge 3$.
\end{conjecture}
\noindent For the first case $e=3$, they were able to show that $1/6\le \pi_d(3,2,3)\le 2/9$.
To the best of our knowledge, for more than forty years no significant improvement was made until recently Glock \cite{glock2018triple} closed the gap by showing that
\begin{equation}\label{glock}
  \begin{aligned}
    \pi_d(3,2,3)=\lim_{n\rightarrow\infty}\fr{f_3(n,5,3)}{n^2}=\fr{1}{5}.
  \end{aligned}
\end{equation}

In this paper we continue this line of research, and  in the spirit of \eqref{order} and Conjecture \ref{besconjecture} we consider the following question.
\begin{question}\label{openquestion}
  For fixed integers $r>k\ge 2, e\ge 3$, does the limit
  $$\pi_d(r,k,e)=\lim_{n\rightarrow\infty}\fr{f_r(n,er-(e-1)k,e)}{n^k}$$ exist?
  If so, what is the value of $\pi_d(r,k,e)$?
\end{question}

In general this question is widely open.
The authors of \cite{BES71} who established \eqref{order} did not try  to optimize the coefficient of $n^k$, however a careful analysis of their lower bound yields to
\begin{equation}\label{inequality}
  \begin{aligned}
    \sqrt[e-1]{\frac{(er-(e-1)k)!}{2\binom{er-(e-1)k}{r}\cdot\binom{\binom{er-(e-1)k}{r}}{e}\cdot(r!)^e}}
    \le\pi_d(r,k,e)\le\frac{e-1}{r\cdots(r-k+1)},
  \end{aligned}
\end{equation}
\noindent where the upper bound follows from the observation  that any $k$-subset of $[n]$ is contained in at most $e-1$ edges of a $\mr{G}_r(er-(e-1)k,e)$-free $r$-graph, implying that $f_r(n,er-(e-1)k,e)\le(e-1)\big(\binom{n}{k}/\binom{r}{k}\big)$.
Note that \eqref{inequality} states in fact lower and upper bounds on  $\liminf_{n\rightarrow\infty}\fr{f_r(n,er-(e-1)k,e)}{n^k}$ and $\limsup_{n\rightarrow\infty}\fr{f_r(n,er-(e-1)k,e)}{n^k}$ respectively,
since it is not known  whether $\pi_d(r,k,e)$ exists.
However, to simplify the notations we keep  \eqref{inequality} in its current form, and in the sequel we will frequently use abbreviations of this type.

Our main results are  introduced in the next two subsections, and they include the determination of $\pi_d(r,2,3)$ for any fixed $r\ge 4$,
and new lower and upper bounds for $\pi_d(r,k,3)$ for any fixed $r>k\ge 3$.

\vspace{5pt}

\noindent {\it\textbf{Notations.}} We use standard asymptotic notations $\Omega(\cdot),\Theta(\cdot),O(\cdot)$ and $o(\cdot)$ as $n\rightarrow\infty$,
where for functions $f=f(n)$ and $g=g(n)$, we write $f=O(g)$ if there is a constant $c_1$ such that $|f|\le c_1|g|$;
we write $f=\Omega(g)$ if there is a constant $c_2$ such that $|f|\ge c_2 |g|$;
we write $f=\Theta(g)$ if $f=O(g)$ and $f=\Omega(g)$ hold simultaneously;
finally, we write $f=o(g)$ if $\lim_{n\rightarrow\infty}(f/g)=0$.

\subsection{The exact value of $\pi_d(r,2,3)$}

\noindent In an $(n,r,2)$-packing, any member of $\binom{[n]}{2}$ is contained in at most one member of $\binom{[n]}{r}$, therefore one can easily verify that such a packing is also $\mr{G}_r(3r-4,3)$-free. This implies that for all fixed $r\ge 4$ the result of R{\"o}dl \cite{Rodlpacking}, written in the above notation, is \begin{equation}
\label{tikitaka}
 \pi_d(r,2,3)\ge\fr{1}{r^2-r}.
\end{equation}
We  will give a tighter bound than \eqref{tikitaka} by showing that approximately, a ($\fr{1}{r^2-r-1}$)-fraction of the $2$-subsets in $[n]$ can be contained in two $r$-subsets, while the resulting hypergraph still has the  $\mr{G}_r(3r-4,3)$-free property (see Remark \ref{remark}). As a consequence, we obtain the following  improvement on the above lower bound.

\begin{theorem}\label{maintheorem}
  For any fixed integer $r\ge 4$,
  $$\pi(r,2,3)=\lim_{n\rightarrow\infty}\fr{f_r(n,3r-4,3)}{n^2}=\fr{1}{r^2-r-1}.$$
\end{theorem}
Note that Theorem \ref{maintheorem} extends \eqref{glock} from $r=3$ to arbitrary fixed $r\ge 4$.
To prove this theorem it suffices to show that $\limsup_{n\rightarrow\infty}\fr{f_r(n,3r-4,3)}{n^2}\le\fr{1}{r^2-r-1}$ and  $\liminf_{n\rightarrow\infty}\fr{f_r(n,3r-4,3)}{n^2}\ge\fr{1}{r^2-r-1}$.
The upper bound is a special case of the upper bound stated in Theorem \ref{secondtheorem} below, which will be discussed later.
The main difficulty in proving Theorem \ref{maintheorem} is the construction which establishes the lower bound.
In what follows we  briefly review the main ideas behind it.

Generally speaking, the lower bound is obtained by a recursive construction (recursion on the uniformity $r$) and a carefully designed approximate induced decomposition of $K_n$,
the complete graph on $n$ vertices.
Given a finite graph $G$, a {\it $G$-packing} in $K_n$ is simply a family of edge disjoint copies of $G$ in $K_n$.
We will make use of the following lemma, which was proved to be very useful in many other combinatorial constructions (see, e.g. \cite{Alon1,Alon3,Frankl1987,FurediCance2012,glock2018triple}).
\begin{lemma}[Graph packing lemma, see Theorem 2.2 \cite{Frankl1987} or Theorem 3.2 \cite{Alon2}]\label{packinglemma}
  Let $G$ be any fixed graph with $e$ edges and $\epsilon>0$ be any small constant.
  Then there is an integer $n_0$ such that for any $n>n_0$, there exists a $G$-packing $\mr{G}=\{G^1,\ldots,G^l\}$ in $K_n$ with
  $$l\ge(1-\epsilon)\fr{n^2}{2e}$$ edge disjoint copies of $G$ such that
  \begin{itemize}
    \item [\rm{(i)}] any two distinct copies of $G$ share at most two vertices, i.e., $|V(G^i)\cap V(G^j)|\le 2$ for any $1\le i\neq j\le l$;
    \item [\rm{(ii)}] if two distinct copies  $G^i, G^j$ share two vertices $a,b$, then $\{a,b\}$ is neither an edge of $G^i$, nor $G^j$.
  \end{itemize}
\end{lemma}
A $G$-packing satisfying (ii) is called an {\it induced} $G$-packing (see, e.g. \cite{Frankl1987}).
Note that a weaker version of the above lemma, which only considered the existence of a large $G$-packing, regardless of the additional properties (i) and (ii), was used in \cite{glock2018triple}
(see Theorem 5 of \cite{glock2018triple}) to prove the lower bound of \eqref{glock}.
It is easy to see that Lemma \ref{packinglemma} is near-optimal in the sense that the maximum size of any $G$-packing in $K_n$ cannot exceed $\binom{n}{2}/e$.

We call the graph $G$ in Lemma \ref{packinglemma} the {\it component graph}, as it forms the basic component in the approximate decomposition.
Following Theorem \ref{maintheorem} it is natural to call a $\mr{G}_r(3r-4,3)$-free $r$-graph $\ma{H}\s\binom{[n]}{r}$ {\it optimal}
if it has roughly $\big(\frac{1}{r^2-r-1}+o(1)\big)n^2$ edges as $n\rightarrow\infty$.

The following construction summarizes the main steps taken to prove the lower bound in Theorem \ref{maintheorem}.
\begin{construction}\label{outlineofConstruction1}
Given  $\ma{H}$, an optimal $\mr{G}_r(3r-4,3)$-free $r$-graph, we construct an optimal $\mr{G}_{r+1}(3(r+1)-4,3)$-free $(r+1)$-graph by performing the  following three steps.
\begin{itemize}
  \item [\rm{(1)}] By applying Lemma \ref{packinglemma} with a carefully designed {\it component graph} $G_t$ (see Subsection \ref{step1}),
  we approximately decompose  the complete graph $K_n$ to $l=(1-\epsilon)n^2/2|G_t|$ edge disjoint copies of $G_t$, say, $G_t^1,G_t^2,\ldots,G_t^l$;
  \item [\rm{(2)}] For $1\le i\le l$,  by embedding in $V(G_t^i)$ many copies of $\ma{H}$  in a suitable way (see Subsection \ref{step2}) we get  an $(r+1)$-graph $G_t^i(\ma{H})$ (see Lemma \ref{F(G_t)});
  \item [\rm{(3)}] Output the $(r+1)$-graph $\ma{F}:=\cup_{i=1}^l G_t^i(\ma{H})$, the edge disjoint union of the $G_t^i(\ma{H})$'s  (see Subsection \ref{step3}).
\end{itemize}
\end{construction}
The base case, i.e., the optimal $\mr{G}_r(3r-4,3)$-free $r$-graph for $r=3$ was given by Glock \cite{glock2018triple}.
Then, by applying Construction \ref{outlineofConstruction1} iteratively, one can construct optimal $\mr{G}_r(3r-4,3)$-free $r$-graphs for all $r\geq 3$.
The reader is referred to Section \ref{fistlower} for more details.

\subsection{New lower and upper bounds for $\pi_d(r,k,3)$}

\noindent In the beginning of the last subsection it was mentioned that an $(n,r,2)$-packing is also a $\mr{G}_r(3r-4,3)$-free $r$-graph.
However, this is not true in general, namely for $r>k\ge 3$, an $(n,r,k)$-packing is not necessarily a $\mr{G}_r(3r-2k,3)$-free $r$-graph,
as $3(k-1)<2k$ if and only if $k<3$.

Our next result provides new lower and upper bounds for $\pi(r,k,3)$ for any fixed $r>k\ge 2$.
\begin{theorem}\label{secondtheorem}
For any fixed integers $r>k\ge 2$,
 $$\frac{1}{r^k-r}\le\liminf_{n\rightarrow\infty}\frac{f_r(n,3r-2k,3)}{n^k}\le\limsup_{n\rightarrow\infty}\frac{f_r(n,3r-2k,3)}{n^k}\le \frac{1}{k!\binom{r}{k}-\frac{k!}{2}}.$$
\end{theorem}
One can easily check that for $r$ much larger than $k$ the gap between the lower and upper bounds in Theorem \ref{secondtheorem} is quite small.
For example, let $r=\frac{k!}{2}$ and $k$ be sufficiently large, then the two bounds almost match, as $r^k\thickapprox k!\binom{r}{k}$.
On the contrary, if $r$ is approximately $k$, the lower bound becomes even weaker than that of \eqref{inequality}. We omit the detailed computation.

The upper bound in Theorem \ref{secondtheorem}, which includes that of Theorem \ref{maintheorem} as a special case,
follows from a weighted counting argument, and is presented in Section \ref{upper}.
The lower bound is proved by an algebraic construction, which relies on a new matrix property called  {\it strongly $3$-perfect hashing}, which is introduced below in Definition \ref{stronghashing}.
The following lemma shows that in order to construct a $\mr{G}_r(3r-2k,3)$-free $r$-graph it is sufficient to construct a matrix with this property.

\begin{lemma}\label{matr}
Let $r>k\ge 2, \text{ and } q$ be integers. If $\ma{M}$ is a strongly $3$-perfect hashing $q$-ary matrix of order $r\times q^k$,
then it induces a $\mr{G}_r(3r-2k,3)$-free $r$-partite $r$-graph $\ma{H}_{\ma{M}}$ over $n=rq$ vertices and $q^k$ edges,
where the vertices can be partitioned to $r$ disjoint parts $V_1,\ldots,V_r$ of size $q$ each.
\end{lemma}

\noindent The proof of Lemma \ref{matr} is given in Subsection \ref{defofperfecthash}.
Indeed, the multipartite $r$-graph constructed using Lemma \ref{matr} is {\it optimal} up to a constant,
in the sense that it is easy to verify by the pigeonhole principle that any $\mr{G}_r(3r-2k,3)$-free $r$-partite $r$-graph, which has equal part size $q$, can have at most $2q^k$ edges.

The next construction outlines the main ingredients in proving the lower bound  of Theorem \ref{secondtheorem}.
\begin{construction}[Construction proving the lower bound of Theorem \ref{secondtheorem}]\label{outlineofConstruction2}
By induction we assume that $f_r(n,3r-2k,3)\ge\frac{n^k}{r^k-r}-an^{k-1}$ holds for every integer less than $n$,
where $a=a(r,k)$ is some constant not depending on $n$, and we  prove the statement for $n$.
\begin{itemize}
  \item [\rm{(1)}] For fixed $r,k$, let $q$ be the largest prime power satisfying $rq\le n$.
  By using the algebraic construction given in Subsections \ref{furedi} and \ref{matrixconstrucion} we obtain an $r\times q^k$ $q$-ary strongly $3$-perfect hashing matrix $\ma{M}$,
  which by Lemma \ref{matr} induces an $r$-partite $r$-graph $\ma{H}_{\ma{M}}$ over $r$ vertex parts $V_1,\ldots,V_r$;
  \item [\rm{(2)}] By the induction hypothesis construct on each vertex part $V_i$ a $\mr{G}_r(3r-2k,3)$-free $r$-graph $\ma{H}_i$ with at least $\frac{q^k}{r^k-r}-aq^{k-1}$ edges;
  \item [\rm{(3)}] Output the $r$-graph $\ma{F}:=(\cup_{i=1}^r\ma{H}_i)\cup\ma{H}_{\ma{M}}$, whose edges are the disjoint union of the edges of $\ma{H}_i, 1\le i\le r$ and $\ma{H}_{\ma{M}}$.
\end{itemize}
\end{construction}

The $r$-graph  $\ma{F}$ has $rq$ vertices and at least
$$q^k+r\cdot(\frac{q^k}{r^k-r}-aq^{k-1})=\frac{(rq)^k}{r^k-r}-arq^{k-1}$$ edges.
In order to complete the induction step it remains to show that $\ma{F}$ is  $\mr{G}_r(3r-2k,3)$-free, and that the number of its edges is at least $\frac{n^k}{r^k-r}-an^{k-1}$.
The detailed proof is given in Section \ref{secondlower}.

\subsection{Outline of the paper}

\noindent The rest of the paper is organized as follows.
In Section \ref{relatedwork} we briefly introduce two combinatorial problems which are closely related to the study of $\pi_d(r,k,e)$.
In Section \ref{upper} we present the proof of the upper bound stated in Theorem \ref{secondtheorem}.
In Sections \ref{fistlower} and \ref{secondlower} we present the proofs of the lower bounds stated in Theorems \ref{maintheorem} and \ref{secondtheorem}, respectively.

\section{Related work}\label{relatedwork}

\subsection{The order of $f_r(n,er-(e-1)k+1,e)$}

\noindent In Question \ref{openquestion} we asked whether $f_r(n,er-(e-1)k,e)/n^k$ converges as $n$ tends to infinity.
In a  similar setting, Brown, Erd\H{o}s and S\'os \cite{BES71} and Alon and Shapira \cite{AlonShapira} posed the following conjecture.
\begin{conjecture}[see, e.g. \cite{BES71,AlonShapira}]\label{related}
  For fixed integers $r>k\ge 2, e\ge 3$, it holds that
  $$n^{k-o(1)}<f_r(n,er-(e-1)k+1,e)=o(n^k)$$ as $n\rightarrow\infty$.
\end{conjecture}
\noindent Note that by \eqref{BESbound}, $$\Omega(n^{k-\fr{1}{e}})<f_r(n,er-(e-1)k+1,e)=O(n^k).$$

Conjecture \ref{related} plays an important role in extremal graph theory.
The first case of the conjecture, namely the determination of the order of $f_3(n,6,3)$, was only resolved by Ruzsa and Szemer\'edi \cite{Ruzsa-Szemeredi} in the  (6,3)-theorem, which was an early application of the celebrated Regularity Lemma \cite{Szemeredi}, while establishing a surprising connection with additive number theory \cite{Behrend46}.
Following efforts of many researchers, the upper bound part of the conjecture is now known to hold for all $r\ge k+1\ge e\ge 3$, \cite{AlonShapira,EPR,ge2017sparse,Rodl,Ruzsa-Szemeredi},
whereas the lower bound is known to hold for all $r>k\ge 2,e=3$ \cite{AlonShapira,EPR,Ruzsa-Szemeredi} and $r>k=2,e\in\{4,5,7,8\}$ \cite{ge2017sparse}.
The  reader is referred to \cite{shangguan2019} for the best known  general lower bound, which shows that for all fixed $r>k\ge 2,e\ge 3$,
$$f_r(n,er-(e-1)k+1,e)=\Omega(n^{k-\fr{1}{e}}(\log n)^{\fr{1}{e-1}}).$$
The smallest case of Conjecture \ref{related} which is still  unresolved is the determination whether $f_3(n,7,4)=o(n^2)$, which is  known as the (7,4)-problem.

\subsection{Locally sparse hypergraphs}

\noindent Recall that  an $(n,r,t)$-packing is of size at most $\binom{n}{t}/\binom{r}{t}$, and such a packing is called an {\it $(n,r,t)$-design} or an {\it $(n,r,t)$-Steiner system}
if its size attains this upper bound with equality.  In \cite{ErdosTuranExponents1} an $(n,3,2)$-design is called $e$-sparse if it is simultaneously $\mr{G}_3(i+2,i)$-free for every $2\le i\le e$.
Erd\H{o}s \cite{ErdosTuranExponents1} posed the following conjecture on the existence of $e$-sparse Steiner triple systems.
\begin{conjecture}[\cite{ErdosTuranExponents1}]\label{erd}
  For a fixed integer $e\ge 2$, there exists $n_0=n_0(e)$ such that one can construct $n$-vertex $e$-sparse Steiner triple systems for every $n\ge n_0$ with $n\equiv 1,3 {\pmod 6}$.
\end{conjecture}
Recent results attained towards resolving  this conjecture were proved independently by Bohman and Warnke \cite{bohman2019large}, and Glock, K\"uhn, Lo and Osthus \cite{glock2018conjecture},
who showed that for fixed $e$, there exist $e$-sparse $(n,3,2)$-packings with size $(1-o(1))n^2/6$, which is near-optimal.
A generalization of Conjecture \ref{erd} was made by F\"{u}redi and Ruszink\'{o} \cite{sparse} (see also Conjecture 7.2 of \cite{glock2018conjecture} for another generalization),
who conjectured the existence of $e$-sparse $(n,r,2)$-Steiner systems, where an $(n,r,2)$-Steiner system is called $e$-sparse if it is simultaneously $\mr{G}_r(ir-2i+2,i)$-free for every $2\le i\le e$.

Generalizing Question \ref{openquestion} in the spirit of the conjectures of Erd\H{o}s, and F\"{u}redi and Ruszink\'{o} leads to the following question.
For fixed integers $r>k\ge 2, e\ge 3$, an $r$-graph is called {\it locally $(e,k)$-sparse} if it is $\mr{G}_r(ir-(i-1)k,i)$-free for every $2\le i\le e$.
\begin{question}
  Let $r>k\ge 2, e\ge 3$ be fixed integers, and $n$ be a sufficiently large integer.
  Then do there exist locally $(e,k)$-sparse $(n,r,k)$-packings with size at least $(1-o(1))\binom{n}{k}/\binom{r}{k}$, where $o(1)\rightarrow 0$ as $n\rightarrow\infty$?
\end{question}

\section{Proof of Theorem \ref{secondtheorem}, the upper bound}\label{upper}

\noindent To prove the upper bound in Theorem \ref{secondtheorem}, we need the following technical lemma.
Let $\ma{H}\s\binom{[n]}{r}$ be an $r$-graph and $T\s[n]$ be a subset.
The {\it codegree} of $T$ in $\ma{H}$, $\deg_{\ma{H}}(T)$, is the number of edges in $\ma{H}$ which contain $T$ as a subset, i.e., $\deg_{\ma{H}}(T)=|\{A\in\ma{H}:T\s A\}|.$
\begin{lemma}\label{naivelemma2}
  An $r$-graph $\ma{H}$ can be made to have no $(k-1)$-subset of codegree one by deleting at most $\binom{n}{k-1}$ of its edges.
\end{lemma}
\begin{proof}
  Successively remove the edges of $\ma{H}$ which contain at least one $(k-1)$-subset of codegree one.
  Let $A_i$ be the $i$-th removed edge of $\ma{H}$, and $T_i$ be some $(k-1)$-subset of codegree one contained in $A_i$.
  Since during this process the codegree of any $(k-1)$-subset can only decrease, then $T_i \neq T_j$ for $i\neq j$.
  In other words, the edges $A_i, A_j$ are removed due to distinct $(k-1)$-subsets of codegree one, and therefore the process terminates after at most $\binom{n}{k-1}$ edge removals.
  Note that the resulting $r$-graph is possibly empty.
\end{proof}

Next we present the proof of the upper bound in Theorem \ref{secondtheorem}.
\vspace{10pt}

\noindent{\it\textbf{Proof of Theorem \ref{secondtheorem}, the upper bound.}}
  Let $\ma{H}\s\binom{[n]}{r}$ be a $\mr{G}_r(3r-2k,3)$-free $r$-graph, and let $\ma{F}$ be the resulting $r$-graph from Lemma \ref{naivelemma2}
  by removing at most $\binom{n}{k-1}$ edges from $\ma{H}$, therefore $|\ma{H}|\leq |\ma{F}|+O(n^{k-1})$.
  The upper bound stated in Theorem \ref{secondtheorem} will follow by showing that
  $$|\ma{F}|\le\fr{2\binom{r}{k}}{2\binom{r}{k}-1}\frac{\binom{n}{k}}{\binom{r}{k}}.$$

  By the $\mr{G}_r(3r-2k,3)$-freeness of $\ma{F}$, it is clear that any $k$-subset of $[n]$ is contained in at most two edges of $\ma{F}$.
  For $i\in\{1,2\}$, let $\ma{K}_i\s\binom{[n]}{k}$ be the family of $k$-subsets of $[n]$ with codegree $i$ in $\ma{F}$, i.e.,
  $$\ma{K}_i=\{K\in\binom{[n]}{k}:\deg_{\ma{F}}(K)=i\}.$$
  Then, for any $A\in\ma{F}$ and $K\in\binom{A}{k}$, either $K\in\ma{K}_1$ or $K\in\ma{K}_2$, and by counting the number of $k$-subsets contained in the edges of $\ma{F}$ it follows that
  \begin{equation}\label{formula1}
    \begin{aligned}
      \binom{r}{k}|\ma{F}|=|\ma{K}_1|+2|\ma{K}_2|.
    \end{aligned}
  \end{equation}
  For  $K=\{x_1,\ldots,x_k\}\in\ma{K}_2$, let  $A,B\in\ma{F}$  be the two edges that contain it, hence  $|A\cap B|\ge k$.  We claim that  in fact $|A\cap B|=k$.
  Indeed, let $a\in A\backslash B$ and consider the $(k-1)$-subset $\{x_1,\ldots,x_{k-2},a\}\subseteq A$.
  As $\ma{F}$ contains no $(k-1)$-subset of codegree one, $\deg_{\ma{F}}(\{x_1,\ldots,x_{k-2},a\})\ge 2$,
  which implies that there exists at least one edge $C\in\ma{F}\setminus\{A,B\}$ such that $\{x_1,\ldots,x_{k-2},a\}\s C$.
  If $|A\cap B|\ge k+1$, then
	\begin{equation}
	|A\cup B\cup C|\le 3r- |A\cap B|-|A\cap C|\le 3r-(k+1)-(k-1)=3r-2k,
	\label{eq:stam}
	\end{equation}
 a contradiction.

 Next we define for a $k$-subset $K\in\ma{K}_2$, and the two distinct $r$-subsets $A, B\in\ma{F}$ containing it, the family of  $k$-subsets $\Phi_{K}:=\big(\binom{A}{k}\cup\binom{B}{k}\big)\setminus\{K\}$. Since $|A\cap B|=k$ we have that
	\begin{equation}
	|\Phi_{K}|=2\binom{r}{k}-2.
	\label{eq:stam2}
	\end{equation}
  Furthermore, by a similar calculation to \eqref{eq:stam} one can verify that
  \begin{equation}\label{phi_L}
    \begin{aligned}
      \Phi_K\s\ma{K}_1.
    \end{aligned}
  \end{equation}

We have the following claim.
  \begin{claim}\label{non-intersection}
$\Phi_{K}\cap \Phi_{K'}=\emptyset$ for distinct $K,K'\in\ma{K}_2$.
  \end{claim}
Assuming the correctness of the claim, together  with \eqref{eq:stam2}, \eqref{phi_L} it follows that
\begin{equation}\label{formula2}
  \begin{aligned}
|\ma{K}_2|(2\binom{r}{k}-2)\leq |\ma{K}_1|.
  \end{aligned}
\end{equation}
It is also easy to see that
\begin{equation}\label{formula3}
  \begin{aligned}
|\ma{K}_1|+|\ma{K}_2|\le\binom{n}{k}.
  \end{aligned}
\end{equation}
Combining \eqref{formula1}, \eqref{formula2} and \eqref{formula3}, we conclude that
\begin{equation*}
  \begin{aligned}
   \binom{r}{k}|\ma{F}|&=|\ma{K}_1|+2|\ma{K}_2|=\fr{2\binom{r}{k}}{2\binom{r}{k}-1}(|\ma{K}_2|+|\ma{K}_1|)+
   \fr{1}{2\binom{r}{k}-1}\big((2\binom{r}{k}-2)|\ma{K}_2|-|\ma{K}_1|\big)\\
   &\le\fr{2\binom{r}{k}}{2\binom{r}{k}-1}(|\ma{K}_2|+|\ma{K}_1|)\le\fr{2\binom{r}{k}}{2\binom{r}{k}-1}\binom{n}{k},\\
  \end{aligned}
\end{equation*}
\noindent
as needed.

It remains to prove Claim \ref{non-intersection}. For the sake of contradiction, assume that there exist two distinct $k$-subsets $K,K'\in\ma{K}_2$ with $\Phi_{K}\cap \Phi_{K'}\neq \emptyset$,
then there is an edge $A\in \ma{F}$ with $K,K'\s A$, otherwise this would contradict the fact that $\Phi_{K}\cap \Phi_{K'} \s \ma{K}_1$ which follows by  \eqref{phi_L}.
Let $B,C\in \ma{F}$ be the edges such that $A\cap B=K$ and $A\cap C=K'$, then
$$|A\cup B\cup C|\leq 3r-|A\cap B|-|A\cap C|=3r-2k,$$
and we arrive at a contradiction, completing the proof of the claim.
$\hfill\square$

\section{Proof of Theorem \ref{maintheorem}}\label{fistlower}

\noindent In this section we prove Theorem \ref{maintheorem}.
By plugging $k=2$ in the upper bound of Theorem \ref{secondtheorem} we get  that $\limsup_{n\rightarrow\infty}\frac{f_r(n,3r-4,3)}{n^2}\le\fr{1}{r^2-r-1}$,
hence it remains to prove the other direction, i.e.,  $\liminf_{n\rightarrow\infty}\frac{f_r(n,3r-4,3)}{n^2}\ge\fr{1}{r^2-r-1}$.

\subsection{The  graph $G_t$}\label{step1}

\noindent In this subsection we define the  graph $G_t$, which is used in step (1) of Construction \ref{outlineofConstruction1} as the component graph of Lemma \ref{packinglemma},
but first we will need the following definition.

\begin{definition}[$\ma{H}$-embedding]\label{H-embedding}
For integers $r\le s, 1\le m\le \binom{s}{r}$, let $\ma{H}=\{A_1,\ldots,A_m\}\s\binom{[s]}{r}$ be an $r$-graph with $m$
edges and $S$ be a set of $s$ elements.
An {\it $\ma{H}$-embedding } from $[s]$ to $S$ is a bijection $\Psi_S:[s]\longrightarrow S$ that acts naturally on the edges of $\ma{H}$ as follows
\begin{equation*}
  \begin{aligned}
   \Psi_S:\ma{H}&\longrightarrow \binom{S}{r}\\
    A\in\ma{H}&\longmapsto\{\Psi_S(a):a\in A\}\in\binom{S}{r}.\\
  \end{aligned}
\end{equation*}
\noindent Clearly, the image of the embedding $\Psi_S(\ma{H}):=\{\Psi_{S}(A_1),\ldots,\Psi_{S}(A_m)\}$ is an $r$-graph  on the vertex set $S$ which forms a copy of $\ma{H}$.
\end{definition}

Let $X=\{x_1,...,x_m\}$ be a set of $m$ elements, and let  $S_1,\ldots,S_t$ be $t$ disjoint sets of size $s$ each, which are also disjoint from the set $X$.
Finally, for  $1\le i\le t$, let $$\Psi_{S_i}(\ma{H})=\{\Psi_{S_i}(A_1),\ldots,\Psi_{S_i}(A_m)\}\s\binom{S_i}{r},$$
be an $\ma{H}$-embedding in $S_i$.

\begin{definition}[The definition of $G_t$]\label{G_tdefinition}
The graph $G_t$, defined on the vertex set
$$V(G_t):=(\cup_{i=1}^t S_i)\cup X$$
of size $ts+m$, is constructed by taking the union of the following three edge sets:
 \begin{itemize}
   \item [\rm{(i)}] $\ma{E}_1=\{\text{edges connecting any two distinct vertices of $S_i,1\le i\le t$}\}$;
   \item [\rm{(ii)}] $\ma{E}_2=\{\text{edges connecting any two distinct vertices of $X$}\}$;
   \item [\rm{(iii)}] $\ma{E}_3=\{\text{edges connecting $x_j$ and each vertex of $\Psi_{S_i}(A_j)$,
   $1\le j\le m$, $1\le i\le t$}\}$.
 \end{itemize}
 \end{definition}
The following two simple observations are crucial for our construction.
\begin{itemize}
	\item The induced subgraph of $G_t$ on each of the $t+1$  subsets $S_1,\ldots,S_t$ and $X$ is the complete graph.
	\item The sets of edges $\ma{E}_i, \ma{E}_j$ are disjoint for $i\neq j$, therefore
\begin{equation}\label{E(G_t)}
  \begin{aligned}
  |G_t|=\sum_{i=1}^3|\ma{E}_i|, \text{ where } |\ma{E}_1|=t\binom{s}{2}, |\ma{E}_2|=\binom{m}{2}, \text{ and }\ma{E}_3=rmt.
  \end{aligned}
\end{equation}
\end{itemize}

\subsection{Lifting the $\ma{H}$-embeddings to an $(r+1)$-graph}\label{step2}

\noindent In this subsection, according to step (2) of Construction \ref{outlineofConstruction1} we lift the $t$ $r$-graphs $\Psi_{S_1}(\ma{H}),\ldots,\Psi_{S_t}(\ma{H})$ that were introduced above, to an $(r+1)$-graph $G_t(\ma{H})$ on $(\cup_{i=1}^t S_i)\cup X$, the vertex set $V(G_t)$.

We call the $t$ $s$-subsets $S_1,\ldots,S_t$ {\it petals}  and the  $m$-subset $X$ the {\it core}.
An $(r+1)$-subset $F\s V(G_t)$ is called {\it $S_i$-rooted} if it contains $r$ vertices of the petal $S_i$, and one vertex of $X$.
In such a case, $r(F):=F\cap S_i$ and $c(F):=F\cap X$ are called the {\it root} and the {\it core} of $F$, respectively.
Let $G_t(\ma{H},S_i)$ be the $(r+1)$-graph on the vertex set $S_i\cup X$, with the following $m$ $S_i$-rooted edges
\begin{equation}\label{G_t(H,S_i)}
  \begin{aligned}
    G_t(\ma{H},S_i)=\big\{\Psi_{S_i}(A_1)\cup\{x_1\},\ldots,\Psi_{S_i}(A_m)\cup\{x_m\}\big\},
  \end{aligned}
\end{equation}
where it is easy to verify that $\{r(F):F\in G_t(\ma{H},S_i)\}$ forms a copy of $\ma{H}$.
Next we define \begin{equation}\label{G_t(H)}
  \begin{aligned}
    G_t(\ma{H})=\cup_{i=1}^t G_t(\ma{H},S_i),
  \end{aligned}
\end{equation}
 to be the edge disjoint\footnote{Since the $S_i$'s are pairwise vertex disjoint,  the $G_t(\ma{H},S_i)$'s are pairwise edge disjoint.} union of the $t$ $(r+1)$-graphs $G_t(\ma{H},S_i).$

The following lemma shows that the $(r+1)$-graph $G_t(\ma{H})$ inherits the \emph{freeness} property from the $r$-graph $\ma{H}$.
More precisely, if $\ma{H}$ is  $\mr{G}_{r}(3r-4,3)$-free then $G_t(\ma{H})$ is  $\mr{G}_{r+1}(3(r+1)-4,3)$-free.
Note that an $r$-graph is called {\it almost linear} if any two distinct edges of it intersect in at most two vertices.
Moreover, it is easy to see that any edge $F\in G_t(\ma{H})$ is $S_i$-rooted for some $1\le i\le t$, and therefore for any (not necessarily distinct) $F_1,F_2\in G_t(\ma{H})$,
\begin{equation}\label{rootdisjointcore}
  \begin{aligned}
    r(F_1)\cap c(F_2)=\emptyset.
  \end{aligned}
\end{equation}
\begin{lemma}\label{F(G_t)} The $(r+1)$-graph $G_t(\ma{H})$ defined in \eqref{G_t(H)} satisfies the following properties:
  \begin{itemize}
    \item [\rm{(i)}] Any two distinct edges of $G_t(\ma{H})$ that are rooted in the same petal have distinct cores;
    \item [\rm{(ii)}] Any two distinct edges of $G_t(\ma{H})$ that are rooted in different petals have disjoint roots;
    \item [\rm{(iii)}] $G_t(\ma{H})$ has  $mt$ $(r+1)$-edges;
    \item [\rm{(iv)}] The vertex set of any $(r+1)$-edge of $G_t(\ma{H})$ induces a complete subgraph in the graph  $G_t$;
    \item [\rm{(v)}] If $\ma{H}$ is almost linear, then $G_t(\ma{H})$ is also almost linear;
    moreover, if $\ma{H}$ is also $\mr{G}_r(3r-4,3)$-free, then $G_t(\ma{H})$ is $\mr{G}_{r+1}(3r-1,3)$-free.
  \end{itemize}
\end{lemma}

\begin{proof}
The first four statements follow easily from the definitions of $G_t, G_t(\ma{H},S_i)$ and $G_t(\ma{H})$.
To prove the first part of (v) we show that $|F_1\cap F_2|\le 2$ for distinct edges  $F_1, F_2\in G_t(\ma{H})$.
If $F_1$ and $F_2$ are rooted in the same petal $S_i$, then since $\{r(F): F\in G_t(\ma{H},S_i)\}$ forms a copy of $\ma{H}$, which is  almost linear, together with \eqref{rootdisjointcore}
we conclude that
$$|F_1\cap F_2|=|\big(r(F_1)\cup c(F_1)\big)\cap\big(r(F_2)\cup c(F_2)\big)|=|r(F_1)\cap r(F_2)|\le 2,$$
as needed.
On the other hand, if $F_1$ and $F_2$ are rooted in different petals, then by \eqref{rootdisjointcore} and (ii) it follows that $F_1\cap F_2=c(F_1)\cap c(F_2)$, implying that $|F_1\cap F_2|\le 1$.

To prove the second part of  (v) consider three distinct  $(r+1)$-edges  $F_1,F_2,F_3\in G_t(\ma{H})$. We have the following three cases:
\begin{itemize}
  \item [$(a)$] If $F_1,F_2,F_3$ are rooted in three distinct petals, then by (ii) $r(F_1),r(F_2),r(F_3)$ are pairwise disjoint, hence
      $$|F_1\cup F_2\cup F_3|= |r(F_1)\cup r(F_2)\cup r(F_3)|+|c(F_1)\cup c(F_2)\cup c(F_3)|\ge 3r+1;$$
  \item [$(b)$] If $F_1,F_2,F_3$ are rooted in two distinct petals, say, $F_1,F_2$ are $S_i$-rooted and $F_3$ is $S_j$-rooted for $ i\neq j$, then
   $$|F_1\cup F_2\cup F_3|\ge |F_1\cup F_2|+|r(F_3)|\ge 3r,$$
      as  $|F_1\cup F_2|\ge 2r$ by the first part of (v), and by \eqref{rootdisjointcore} and (ii), $r(F_3)$ is disjoint from $F_1\cup F_2$;
  \item [$(c)$] If $F_1,F_2,F_3$ are rooted in the same petal, say $S_i$,  then since $\{r(F):F\in G_t(\ma{H},S_i)\}$ is a copy of $\ma{H}$ in $S_i$,
  then it follows from \eqref{rootdisjointcore}, (i) and the $\mr{G}_r(3r-4,3)$-freeness of $G_t(\ma{H},S_i)$ that
$$|F_1\cup F_2\cup F_3|= |r(F_1)\cup r(F_2)\cup r(F_3)|+|c(F_1)\cup c(F_2)\cup c(F_3)|\ge 3r,$$ as needed.
\end{itemize}
\end{proof}

\subsection{Constructing $G_t$-packings in a large complete graph}\label{step3}

\noindent Following step (3) of Construction \ref{outlineofConstruction1}, we introduce below the key idea of the recursive construction.

By applying Lemma \ref{packinglemma} with the graph $G_t$ and large enough $n$, one obtains a $G_t$-packing in $K_n$, denoted by $\mr{G}=\{G^1_t,\ldots,G^l_t\}$, which contains roughly $l=(1-o(1))n^2/|G_t|$ edge disjoint copies of $G_t$.
For each $1\le i\le l$, construct on $V(G^i_t)$ the $(r+1)$-graph $G_t^i(\ma{H})$, as defined in \eqref{G_t(H)}. Let
\begin{equation}\label{stamstam}
\ma{F}:=\cup_{i=1}^l G_t^i(\ma{H})
\end{equation} be the union of those $l$ $(r+1)$-graphs.
Recall that Lemma \ref{F(G_t)} implies that if $\ma{H}$ has the freeness property then so has  $G_t^i(\ma{H})$ for each $1\le i\le l$.
The next lemma shows that $\ma{F}$ {\it preserves} the $\mr{G}_{r+1}(3r-1,3)$-freeness of the $G_t^i(\ma{H})$'s.

\begin{lemma}\label{mainlemma}
  Let  $r\ge 3$ be an integer, and $\ma{H}\s\binom{[s]}{r}$ be  an almost linear $\mr{G}_r(3r-4,3)$-free $r$-graph with $m$ edges.
  Then $\ma{F}$  defined in \eqref{stamstam} is almost linear, $\mr{G}_{r+1}(3r-1,3)$-free, with  $mtl$ edges.
\end{lemma}

\begin{proof}
We begin by showing that $|\ma{F}|=mtl$.
It is enough to prove that $\ma{F}$ is an edge disjoint union of the $G_t^i(\ma{H})$'s, $1\le i\le l$.
Indeed, by Lemma \ref{F(G_t)} (iv), each $(r+1)$-edge in $G_t^i(\ma{H})$ induces a complete subgraph in $G_t^i$, implying that for any $1\le i\neq j\le l$, $G_t^i(\ma{H})$ and $G_t^j(\ma{H})$ cannot have any common $(r+1)$-edge, since otherwise $G_t^i$ and $G_t^j$ would have a common 2-edge, contradicting the definition of a $G_t$-packing.

For the almost linearity we prove the following stronger claim.

\begin{claim}\label{almostlinear}
For any two distinct $(r+1)$-edges $F_1, F_2\in \ma{F}$, if there exists an $1\le i\le l$ such that $\{F_1,F_2\}\s G_t^i(\ma{H})$, then $|F_1\cap F_2|\le 2$; otherwise $|F_1\cap F_2|\le 1$.
\end{claim}

The first case of the claim follows easily from the almost linearity of $\ma{H}$ and Lemma \ref{F(G_t)} (v).
To prove the second case, suppose there exist $1\le i\neq j\le l$ such that $F_1\in G_t^i(\ma{H})$ and $F_2\in G_t^j(\ma{H})$.
By Lemma \ref{F(G_t)} (iv) $F_1$ (resp. $F_2$) induces a complete graph on $V(G_t^i)$ (resp. $V(G_t^j)$).
On the other hand, by construction $G_t^i$ and $G_t^j$ are edge disjoint, hence clearly $|F_1\cap F_2|\le 1$.

Next we show that $\ma{F}$ is $\mr{G}_{r+1}(3r-1,3)$-free.
Assume to the contrary that there exist three distinct $(r+1)$-edges $F_1,F_2,F_3\in\ma{F}$ such that $|F_1\cup F_2\cup F_3|\le 3r-1$.
Hence, in such a case there exist $1\le i\neq j\le 3$ such that $|F_i\cap F_j|\ge 2$, since otherwise $|F_1\cup F_2\cup F_3|\ge 3r$.
Without loss of generality, assume that $|F_1\cap F_2|\ge 2$.
Thus it follows from Claim \ref{almostlinear} that there exists $1\le i\le l$ such that $F_1,F_2\in G_t^i(\ma{H})$, and we actually have $|F_1 \cap F_2|=2$, as $G_t^i(\ma{H})$ is almost linear.
We claim that $F_3$ also belongs to $G_t^i(\ma{H})$.
Then given the $\mr{G}_r(3r-4,3)$-freeness of $\ma{H}$, we arrive at a contradiction by Lemma \ref{F(G_t)} (v).

For the sake of contradiction, assume that  $F_3\in G_t^j(\ma{H})$ for  $j\neq i$.
By the inclusion-exclusion principle, it is easy to check that $|F_3\cap(F_1\cup F_2)|\ge 2$, which implies that $|V(G_t^j)\cap V(G_t^i)|\ge 2$.
By Lemma \ref{packinglemma} (i) it follows  that $|V(G_t^j)\cap V(G_t^i)|=2$.
Let $A=V(G_t^j)\cap V(G_t^i)$, then clearly $A\s F_3$, and by Lemma \ref{F(G_t)} (iv)  $A$ forms an edge in $G_t^j$,
which contradicts Lemma \ref{packinglemma} (ii).
\end{proof}

\subsection{Establishing the lower bound of $\pi_d(r,2,3)$}

\noindent To prove the lower bound $\liminf_{n\rightarrow\infty}\fr{f_r(n,3r-4,3)}{n^2}\ge\fr{1}{r^2-r-1}$ it suffices to show that for any  $\epsilon>0$, there exists an integer $n(\epsilon)>0$ such that
for any $n>n(\epsilon)$, there exists an  almost linear $\mr{G}_r(3r-4,3)$-free $r$-graph on $n$ vertices with at least $(\fr{1}{r^2-r-1}-\epsilon)n^2$ $r$-edges.
We will prove this statement by induction on the uniformity parameter $r\geq 3$.

The base case $r=3$ follows directly from the work of Glock \cite{glock2018triple}.
Next, assume that the statement holds for $r\ge 3$, and we prove it for $r+1$.
Given $\epsilon>0$, let $\delta=\delta(\epsilon)>0$ be a sufficiently small constant.
By the induction hypothesis, given $\delta$ there exists an integer $s(\delta)>0$ such that for any $s>s(\delta)$,
there exists  an almost linear $\mr{G}_r(3r-4,3)$-free $r$-graph $\ma{H}\s\binom{[s]}{r}$ with $m$ edges, where
\begin{equation}\label{m}
  \begin{aligned}
    m\ge(\fr{1}{r^2-r-1}-\delta)s^2.
  \end{aligned}
\end{equation}
Let $t$ be a sufficiently large integer satisfying
\begin{equation}\label{tau}
  \begin{aligned}
    \fr{m}{t}<\delta.
  \end{aligned}
\end{equation}
By applying Lemmas \ref{packinglemma} and \ref{mainlemma} with the component graph $G_t$ given in Definition \ref{G_tdefinition}, one can construct an almost linear $\mr{G}_{r+1}(3r-1,3)$-free $(r+1)$-graph $\ma{F}=\cup_{i=1}^l G_t^i(\ma{H})$
on $n>n_0$\footnote{Here $n_0$ is a constant given by applying Lemma \ref{packinglemma} with $G:=G_t$ and $\epsilon:=\delta$.} vertices, with $|\ma{F}|=tml$ edges, where
 \begin{equation}
 \label{sigma}
 l\ge\fr{(1-\delta)n^2}{2(t\binom{s}{2}+\binom{m}{2}+rmt)}	.
\end{equation}
Hence, 
\begin{equation}\label{E(F)}
  \begin{aligned}
  |\ma{F}|&\ge tm\cdot\fr{(1-\delta)n^2}{2(t\binom{s}{2}+\binom{m}{2}+rmt)}\ge\fr{(1-\delta)n^2}{s^2/m+m/t+2r}\geq
	\fr{(1-\delta)n^2}{\fr{r^2-r-1}{1-(r^2-r-1)\delta}+\delta+2r},\\
  \end{aligned}
\end{equation}
where the last inequality follows by  \eqref{m} and \eqref{tau}.
For $\delta=0$ the right hand side of \eqref{E(F)} is greater than
 \begin{equation*}
  \begin{aligned}
   \fr{(1-\epsilon)n^2}{r^2+r-1}=\fr{(1-\epsilon)n^2}{(r+1)^2-(r+1)-1}.
  \end{aligned}
\end{equation*}
Then by continuity there exists  $\delta>0$ for which $\ma{F}$ has at least $\fr{(1-\epsilon)n^2}{(r+1)^2-(r+1)-1}$ edges, completing the induction.
$\hfill\square$

\begin{remark}\label{remark}
Let $\ma{H}\s\binom{[n]}{r}$ be a $\mr{G}_r(3r-4,r)$-free $r$-graph with $\frac{(1-o(1))n^2}{r^2-r-1}$ edges, where $o(1)\rightarrow 0$ as $n\rightarrow\infty$. It is clear by definition that $\ma{H}$ contains no $2$-subsets with codegree larger than $2$. In fact, it follows from \eqref{formula2} and \eqref{formula3} (with $k=2$) that the proportions of $2$-subsets in $[n]$ with codegree $0,1,2$ must be (approximately, as $n\to\infty$) $0,1-\fr{1}{r^2-r-1},\fr{1}{r^2-r-1}$, respectively.
\end{remark}

\section{Proof of Theorem \ref{secondtheorem}, the lower bound}\label{secondlower}
\subsection{From strongly perfect hashing matrices to sparse hypergraphs}\label{defofperfecthash}

\noindent In this subsection we define the notion of a strongly $3$-perfect hashing  matrix, and show that
any such matrix gives rise to a sparse hypergraph with relatively many edges (see  Lemma \ref{matr}).

We begin with some notations.  Let $\ma{M}$ be an $r\times m$ matrix over $Q$, an alphabet of size $q$, and for $1\le j\le m$  let
$$\vec{c}_j=(c_{1,j},\ldots,c_{r,j})^T\in Q^r,$$
be the $j$-th column of $\ma{M}$.
We say that the $i$-th row of $\ma{M}$ {\it separates} a subset of columns $T$, if the entries of row $i$ restricted to columns in $T$ are all distinct,
i.e., $\{c_{i,j}:\vec{c}_j\in T\}$ is a set of $|T|$ distinct elements of $Q$.

A matrix is called {\it $3$-perfect hashing} (see,  e.g. \cite{Fredman84}) if any three distinct columns of it are separated by at least one row.
In the literature, matrices with different perfect hashing properties have been studied extensively.
The reader is referred to \cite{shangguanperfecthash} and the references therein for a detailed introduction to this topic.
Here we introduce a slightly stronger notion which we term \emph{strongly} $3$-perfect hashing, but first we will need the following notation.
For  $t$ columns $\vec{c}_1,\ldots,\vec{c}_t$ of $\ma{M}$, let $I(\vec{c}_1,\ldots,\vec{c}_t)\s[r]$ denote the collection of row indices for which $\vec{c}_1,\ldots,\vec{c}_t$ have equal entries,
i.e., $i\in I(\vec{c}_1,\ldots,\vec{c}_t)$ if and only if $c_{i,1}=\cdots=c_{i,t}$.

\begin{definition}\label{stronghashing}
  An $r\times q^k$ matrix $\ma{M}$ over $Q$ is called {\it strongly $3$-perfect hashing} if any three distinct columns $\vec{c}_1,\vec{c}_2,\vec{c}_3$ of $\ma{M}$  are separated by \emph{more} than $r-2k+|I(\vec{c}_1,\vec{c}_2,\vec{c}_3)|$ rows.
\end{definition}

Clearly, this definition holds trivially if $r-2k+|I(\vec{c}_1,\vec{c}_2,\vec{c}_3)|< 0$.
However, if for  any three distinct columns $\vec{c}_1,\vec{c}_2,\vec{c}_3$, $r-2k+|I(\vec{c}_1,\vec{c}_2,\vec{c}_3)|\geq 0,$
then a \emph{strongly} $3$-perfect hashing matrix is also a $3$-perfect hashing matrix, justifying the name of this property.


\begin{observation}\label{obs}
Any $r\times m$ matrix $\ma{M}$ over $Q$ defines an $r$-partite $r$-graph $\ma{H}_{\ma{M}}$ with $rq$ vertices and $m$ edges as follows.
The vertex set admits a partition $V(\ma{H}_{\ma{M}})=\cup_{i=1}^r V_i$, where
$$V_i=\{(i,\alpha):\alpha\in Q\}$$
is the $i$-th vertex part of size $q$.
The $m$ edges of $\ma{H}_{\ma{M}}$ are defined by the $m$ columns of $\ma{M}$ as follows
$$A_j:=\{(1,c_{1,j}),\ldots,(r,c_{r,j})\}\s(\cup_{i=1}^r V_i), 1\le j\le m.$$
It is easy to verify that for any edge and any vertex part we have $|A_j|=r, |V_i|=q \text{ and }|A_j\cap V_i|=1$.
\end{observation}

In the remaining part of this section we view matrices (resp. columns of the matrices) and
multipartite hypergraphs (resp. edges of the hypergraphs) as equivalent objects.

Next we present the proof of Lemma \ref{matr}, but first note that  for any $1\le i,j,l\le m$,
$$|I(\vec{c}_{i},\vec{c}_{j},\vec{c}_{l})|=|A_{i}\cap A_j \cap A_l| \text{\qquad and \qquad}
|I(\vec{c}_{i},\vec{c}_{j})=|A_{i}\cap A_{j}|.$$

\vspace{10pt}

\noindent{\it\textbf{Proof of Lemma \ref{matr}.}}
 Given an $r\times q^k$  strongly $3$-perfect hashing matrix $\ma{M}$ over $Q$,  let $\ma{H}_{\ma{M}}=\{A_1,\ldots,A_{q^k}\}$   be the corresponding  $r$-partite $r$-graph with $rq$ vertices and $q^k$ edges, given by Observation \ref{obs}.
 We claim that  $\ma{H}_{\ma{M}}$ is $\mr{G}_r(3r-2k,3)$-free, i.e., for any three distinct edges $A_{i},A_{j},A_{l}\in\ma{H}_{\ma{M}}$,
  $$|A_{i}\cup A_j \cup A_l|> 3r-2k.$$
Since $\ma{M}$ is  strongly $3$-perfect hashing, the columns  $\vec{c}_{i},\vec{c}_{j},\vec{c}_{l}$ are separated by more than
$r-2k+|I(\vec{c}_{i},\vec{c}_{j},\vec{c}_{l})|$ rows. Equivalently, $\vec{c}_{i},\vec{c}_{j},\vec{c}_{l}$ are not separated by less than $2k-|I(\vec{c}_{i_1},\vec{c}_{i_2},\vec{c}_{i_3})|$ rows.
	Hence,
	\begin{equation*}
    \begin{aligned}
      |A_{i}\cup A_j \cup A_l|&\ge 3(r-2k+|I(\vec{c}_{i},\vec{c}_{j},\vec{c}_{l})|+1)+2(2k-|I(\vec{c}_{i},\vec{c}_{j},\vec{c}_{l})|-1)
      -|I(\vec{c}_{i},\vec{c}_{j},\vec{c}_{l})|\\
      &=3r-2k+1,\\
    \end{aligned}
  \end{equation*} as desired.
$\hfill\square$

In order to construct matrices satisfying this useful property we introduce next a technical lemma of F\"uredi \cite{FurediCance2012}.

\subsection{A technical lemma of F\"uredi}\label{furedi}

\noindent In this subsection, we introduce a lemma of F\"uredi \cite{FurediCance2012} on a generalized linear independence property of polynomials.
Let us begin with some necessary terminology.

For a prime power $q$ and a positive integer $k$, let $\mathbb{F}_q$ be the finite field with $q$ elements, and $\mathbb{F}_q^{<k}[x]$ be the set of polynomials of degree less than $k$, with coefficients in $\mathbb{F}_q$.
Clearly $|\mathbb{F}_q^{<k}[x]|=q^k$.
Let $p_1(x),p_2(x),p_3(x)\in\mathbb{F}_q^{<k}[x]$ be three arbitrary polynomials, and $k_1,k_2,k_3$ be positive integers such that $$\sum_{i=1}^3 k_i\le k \text{\quad and\quad} k_i\le k-\deg(p_i)\quad\text{for each $1\le i\le 3$.}$$
The polynomials $p_1(x),p_2(x),p_3(x)$ are said to be {\it $(k_1,k_2,k_3)$-independent,} if for any $q_i(x)\in\mathbb{F}_q^{<k_i}[x],1\le i\le 3$, the equality
$$q_1(x)p_1(x)+q_2(x)p_2(x)+q_3(x)p_3(x)\equiv 0\in \mathbb{F}_q^{<k}[x]$$
holds if and only if each $q_i(x)$ is the zero polynomial in $\mathbb{F}_q^{<k}[x]$.
Equivalently, all the $q^{\sum_{i=1}^3 k_i}$ polynomials of the form $\sum_{i=1}^3 q_i(x)p_i(x)$ are distinct in $\mathbb{F}_q^{<k}[x]$.
Note that the case $k_1=k_2=k_3=1$ reduces to the usual $\mathbb{F}_q$-linear independence of three polynomials in $\mathbb{F}_q^{<k}[x]$.

A vector $\vec{v}=(\alpha_1,\ldots,\alpha_r)\in\mathbb{F}_q^r$ is called {\it nonrepetitive} if all of its entries are pairwise distinct, i.e., $\alpha_i\neq \alpha_j$ for $ i\neq j$.
Given such a vector $\vec{v}\in\mathbb{F}_q^r$ with $r\geq k$, let $\mathbb{F}_q^{<k}[x,\vec{v}]$ be the $k$-dimensional subspace of $\mathbb{F}_q^{r}$ defined as
\begin{equation}\label{subspace}
  \begin{aligned}
    \mathbb{F}_q^{<k}[x,\vec{v}]:=\Big\{\vec{c}_f=\big(f(\alpha_1),\ldots,f(\alpha_r)\big)^T\in\mathbb{F}_q^r:~f\in\mathbb{F}_q^{<k}[x]\Big\}.
  \end{aligned}
\end{equation}
It is not too difficult to verify that $\mathbb{F}_q^{<k}[x,\vec{v}]$ is indeed a $k$-dimensional subspace of $\mathbb{F}_q^{r}$.
Given a set $X\s[r]$ of indices, we define the {\it annihilator} polynomial:
$$p_X(x,\vec{v})=\prod_{i\in X}(x-\alpha_i).$$

The following lemma is crucial for Construction \ref{outlineofConstruction2}, and it was proved by F\"uredi \cite{FurediCance2012}.

\begin{lemma}[see Lemma 10.3 and Corollary 10.4 \cite{FurediCance2012}]\label{lowerfirst}
  Let $k$ be a positive integer and $q$ be a prime power.
  Then for all but at most $k(k-1)q^{2k-1}$ nonrepetitive vectors $\vec{v}=(\alpha_1,\ldots,\alpha_{2k})\in\mathbb{F}_q^{2k}$,
  the polynomials $p_{Z_1}(x,\vec{v}),p_{Z_2}(x,\vec{v}),p_{Z_3}(x,\vec{v})$ are
  $$(k-|Z_1|,k-|Z_2|,k-|Z_3|)\text{-independent,}$$
  for every partition $[2k]=Z_1\cup Z_2\cup Z_3$, with $1\le |Z_i|<k$ for $i=1,2,3$.
\end{lemma}

The following facts are easy to verify.

\begin{fact}\label{simplelemma}
  Let $\vec{v}=(\alpha_1,\ldots,\alpha_r)\in\mathbb{F}_q^r$ be a nonrepetitive vector and $X\s[r]$. If the  polynomials $f_1,f_2$ satisfy $f_1(\alpha_i)=f_2(\alpha_i)$ for each $i\in X$, then
  $$p_X(x,\vec{v}) \mid f_1-f_2.$$
\end{fact}


\begin{fact}\label{mds}
 Two distinct polynomials of degree less than $k$ can agree in at most $k-1$ points in $\mathbb{F}_q$.
\end{fact}


\subsection{Constructing strongly $3$-perfect hashing matrices}\label{matrixconstrucion}
\noindent In this subsection we show the existence of strongly $3$-perfect hashing matrices over large enough finite fields (see Lemma \ref{stronghashing-construction}).

We view the  $k$-dimensional subspace $\mathbb{F}_q^{<k}[x,\vec{v}]$ defined in \eqref{subspace} as an  $r\times q^k$ matrix whose columns are labeled by the polynomials in $\mathbb{F}_q^{<k}[x]$, such that a column with index $f\in\mathbb{F}_q^{<k}[x]$ is the vector $\vec{c}_f$ defined in \eqref{subspace}.

\begin{lemma}\label{stronghashing-construction}
  Let $r>k\ge 2$ be fixed integers and $q>4^r k^2$ be a prime power.
  Then for at least  $q^r-4^rk^2q^{r-1}$ vectors $\vec{v}\in\mathbb{F}_q^r$,
  the matrix $\mathbb{F}_q^{<k}[x,\vec{v}]$ is strongly $3$-perfect hashing.
\end{lemma}

\begin{proof}
Call a vector $\vec{v}\in\mathbb{F}_q^r$ {\it bad} if the corresponding matrix $\mathbb{F}_q^{<k}[x,\vec{v}]$ is not strongly $3$-perfect hashing, and observe that there are at most $\binom{r}{2}q^{r-1}$ vectors in $\mathbb{F}_q^r$ with at least two repeated entries.
To prove the lemma it suffices to show that the number of bad nonrepetitive vectors in $\mathbb{F}_q^r$ is at most $4^rk(k-1)q^{r-1}$, and therefore the total number of bad vectors is bounded from above by  $\binom{r}{2}q^{r-1}+4^rk(k-1)q^{r-1}\le 4^rk^2q^{r-1}$.

We say that a vector $\vec{v}\in\mathbb{F}_q^r$ is bad for a set of row indices $I\s[r]$, if there exist three distinct columns $\vec{c}_1,\vec{c}_2,\vec{c}_3$ of $\mathbb{F}_q^{<k}[x,\vec{v}]$  such that $I=I(\vec{c}_1,\vec{c}_2,\vec{c}_3)$ and they  are separated by at most $r-2k+|I|\geq 0$ rows.
Clearly, a vector $\vec{v}\in\mathbb{F}_q^r$ is bad if it is bad for  some set $I\s[r]$.
For a given subset $I\s[r]$, below we give an upper bound on the number of  nonrepetitive vectors that are bad for it.

Assume that  $\vec{v}\in\mathbb{F}_q^r$ is a nonrepetitive vector which is bad for $I\s[r]$, and let
$\vec{c}_1,\vec{c}_2,\vec{c}_3\in\mathbb{F}_q^{<k}[x,\vec{v}]$ be three distinct columns that violate the strongly $3$-perfect hashing property.
Let $$X:=I(\vec{c}_1,\vec{c}_2)\cup I(\vec{c}_2,\vec{c}_3)\cup I(\vec{c}_1,\vec{c}_3)$$ be the set of rows for which at least two  columns  attain the same value, and therefore $I\subseteq X$.
Since the three columns  are separated by any row whose index is not in  $X$, then  $r-|X|\le r-2k+|I|$, namely,
\begin{equation}
|X|\ge 2k-|I|.
\label{eq:234}
\end{equation}
By Fact \ref{mds} $|I|\leq k-1$. If $|I|=k-1$, then again by Fact \ref{mds} $I=X$ which contradicts \eqref{eq:234}. Therefore, we assume that $|I|<k-1$.
Note that by \eqref{eq:234} $|X\setminus I|=|X|-|I|\ge 2k-2|I|>0$, and let $Y\s X\setminus I$ be an arbitrary subset of  size $2k-2|I|$.
Next, define the following three sets $Z_i$ that form a partition of $Y$:
$$Z_1=\{i\in Y:i\in I(\vec{c}_1,\vec{c}_2)\},~Z_2=\{i\in Y:i\in I(\vec{c}_2,\vec{c}_3)\}
\text{ and }Z_3=\{i\in Y:i\in I(\vec{c}_1,\vec{c}_3)\}.$$
The sets $Z_i$ satisfy  the following claim.
\begin{claim}\label{connection}
For $i=1,2,3$, $1\leq |Z_i|<k-|I|$ and the polynomials $p_{Z_1}(x,\vec{v}),p_{Z_2}(x,\vec{v}),p_{Z_3}(x,\vec{v})$ are not
    $$(k-|I|-|Z_1|,k-|I|-|Z_2|,k-|I|-|Z_3|)\text{-independent}.$$
\end{claim}

\noindent{\it Proof of Claim \ref{connection}.}
Since $k-|I|\geq 2$ the inequalities on the sizes of the sets $Z_i$ are well-defined. The sets $I$ and $Z_1$  are disjoint and are subsets of  $I(\vec{c}_1,\vec{c}_2)$, therefore $|Z_1|+|I|\leq |I(\vec{c}_1,\vec{c}_2)|<k$ which implies the upper bound.  The proof for $Z_2,Z_3$ is the same. Furthermore, if one of the $Z_i$'s is the empty set, say $Z_3$, then this would imply that
$2k-2|I|=|Y|=|Z_1|+|Z_2|$, i.e., either $Z_1$ or $Z_2$ is of size at least $k-|I|$, which is a contradiction.

Assume that $\vec{c}_1,\vec{c}_2,\vec{c}_3$ are indexed by polynomials $f_1,f_2,f_3\in\mathbb{F}_q^{<k}[x]$, respectively.
  Since
  \begin{equation*}
    \begin{aligned}
      (Z_1\cup I)\s I(\vec{c}_1,\vec{c}_2),~(Z_2\cup I)\s I(\vec{c}_2,\vec{c}_3) \text{ and } (Z_3\cup I)\s I(\vec{c}_1,\vec{c}_3),
    \end{aligned}
  \end{equation*}
  then by Fact \ref{simplelemma},
\begin{equation}\label{eq1}
  \begin{aligned}
        &p_{Z_1}(x,\vec{v})p_{I}(x,\vec{v})=p_{Z_1\cup I}(x,\vec{v})\mid f_1-f_2,\\
        &p_{Z_2}(x,\vec{v})p_{I}(x,\vec{v})=p_{Z_2\cup I}(x,\vec{v})\mid f_2-f_3,\\
        &p_{Z_3}(x,\vec{v})p_{I}(x,\vec{v})=p_{Z_3\cup I}(x,\vec{v})\mid f_3-f_1,\\
  \end{aligned}
\end{equation}
which implies that there exist nonzero polynomials $q_i(x)\in\mathbb{F}_q^{<k-|I|-|Z_i|}[x]$ for $1\le i\le 3$, such that
\begin{equation}\label{eq5}
  \begin{aligned}
     &q_1(x)p_{Z_1}(x,\vec{v})p_{I}(x,\vec{v})=f_1-f_2,\\
     &q_2(x)p_{Z_2}(x,\vec{v})p_{I}(x,\vec{v})=f_2-f_3, \\
     &q_3(x)p_{Z_3}(x,\vec{v})p_{I}(x,\vec{v})= f_3-f_1.\\
  \end{aligned}
\end{equation}
By summing the left and right hand sides of \eqref{eq5} we conclude  that
\begin{equation}
  \begin{aligned}
    \big(q_1(x)p_{Z_1}(x,\vec{v})+q_2(x)p_{Z_2}(x,\vec{v})+q_3(x)p_{Z_3}(x,\vec{v})\big)\cdot p_{I}(x,\vec{v})=0.
  \end{aligned}
\end{equation}
The ring of polynomials $\mathbb{F}_q[x]$ is a domain, implying that
$$q_1(x)p_{Z_1}(x,\vec{v})+q_2(x)p_{Z_2}(x,\vec{v})+q_3(x)p_{Z_3}(x,\vec{v})=0,$$
namely, $p_{Z_1}(x,\vec{v}),p_{Z_2}(x,\vec{v}),p_{Z_3}(x,\vec{v})\text{ are not }(k-|I|-|Z_1|,k-|I|-|Z_2|,k-|I|-|Z_3|)\text{-independent}$, completing the proof of the claim. $\hfill\square$

\vspace{10pt}

\noindent{\it\textbf{Continuing the proof of Lemma \ref{stronghashing-construction}:}}
If  $\vec{v}\in\mathbb{F}_q^r$ is a nonrepetitive vector which is bad for $I$, then there exists a $(2k-2|I|)$-subset $Y\s[r]\setminus I$ for which the assertion of Claim \ref{connection} holds. However, Lemma \ref{lowerfirst} provides an upper bound on the number of such vectors $\vec{v}\mid_{Y}\in\mathbb{F}_q^{2k-2|I|}$, where $\vec{v}\mid_{Y}$ is the restriction of $\vec{v}$ to the coordinates in $Y$. More precisely, by Lemma \ref{lowerfirst} there are at most
$$(k-|I|)(k-|I|-1)q^{2k-2|I|-1}\le k(k-1)q^{2k-2|I|-1}$$
possible choices for $\vec{v}\mid_{Y}$.
Thus,  given the sets $I$ and $Y$, the number of nonrepetitive vectors $\vec{v}\in\mathbb{F}_q^r$
which are bad for $I$ is at most
$$q^{|I|}\times k(k-1)q^{2k-2|I|-1}\times q^{r-|I|-(2k-2|I|)}=k(k-1)q^{r-1}.$$
Indeed, there are at most $q^{|I|}$ ways to pick $\vec{v}\mid_{I}$, at most $k(k-1)q^{2k-2|I|-1}$ ways to pick $\vec{v}\mid_{Y}$,
and at most $q^{r-|I|-(2k-2|I|)}$ ways to pick the remaining entries in $\vec{v}\mid_{[r]\setminus (I\cup Y)}$.

Since $I$ and $Y$ are  subsets of $[r]$, then there are at most $(2^r)^2$ ways to choose them.
To conclude, the total number of bad nonrepetitive vectors is at most $4^rk(k-1)q^{r-1}$, as desired.
\end{proof}

\subsection{Establishing the lower bound of $\pi_d(r,k,3)$}\label{induction}

\noindent Let $r>k\ge 2$ be fixed integers, in this subsection we will present the proof of the lower bound $\frac{1}{r^k-r}\le\liminf_{n\rightarrow\infty}\frac{f_r(n,3r-2k,3)}{n^k}$, finishing the comments after Construction \ref{outlineofConstruction2}.
We need the following result on the distribution of primes.

\begin{lemma}[see Theorem 1 \cite{BakerPrimeNumber}]\label{prime}
  There exists a positive integer $n_0$ such that for any integer $n>n_0$, the largest prime $q$ not exceeding $n$ satisfies $q\ge n-n^{\delta}$, where $0<\delta\le 0.525$.
\end{lemma}

The desired lower bound is a straightforward consequence of the following claim.

\begin{claim}\label{lastclaim}
  Let $r>k\ge 2$ be fixed integers and $n_0,\delta$ be the constants given by Lemma \ref{prime}.
  Then there exists a positive constant $a=a(r,k,n_0)$ such that for any $n\ge 1$ there exists  a $\mr{G}_r(3r-2k,3)$-free $r$-graph $\ma{F}$ on $n$ vertices with at least $$\frac{n^k}{r^k-r}-an^{k-1+\delta}$$ edges, such that for any distinct $A,B\in\ma{F}$, $|A\cap B|\le k-1$.
\end{claim}

\begin{proof}
  We will prove the claim by induction on $n$.
  Let $n^*=n^*(r,k,n_0)$ be the smallest $n$ satisfying
  $$\frac{n}{r}> n_0\text{\quad and \quad}\frac{n}{r}-(\frac{n}{r})^{\delta}>4^rk^2,$$
and  pick large enough $a=a(r,k,n_0)$ such that $\frac{n^k}{r^k-r}-an^{k-1+\delta}<0$ for  $n\leq n^*$. For such a choice of $a$  the claim holds trivially for $n\leq n^*$.
  Let $n>n^*$ and $q$ be the largest prime not exceeding $\frac{n}{r}$, then by Lemma \ref{prime}  $q\ge\frac{n}{r}-(\frac{n}{r})^{\delta}>4^rk^2.$
  Thus by Lemma \ref{stronghashing-construction} there exists a vector $\vec{v}\in\mathbb{F}_q^r$ such that the $r\times q^k$ matrix $\ma{M}:=\mathbb{F}_q^{<k}[x,\vec{v}]$ defined by \eqref{subspace}, is strongly $3$-perfect hashing.
  By Lemma \ref{matr}, $\ma{M}$ induces a $\mr{G}_r(3r-2k,3)$-free $r$-partite $r$-graph $\ma{H}_{\ma{M}}$, with $q^k$ edges and $rq$ vertices, such that $V(\ma{H}_{\ma{M}})$ is partitioned into $r$ disjoint parts $V_1,\ldots,V_r$ with $|V_1|=\cdots=|V_r|=q$.
  Since the edges of $\ma{H}_{\ma{M}}$ are defined by the columns of $\ma{M}$, and the columns of $\ma{M}$ are essentially defined by polynomials of degree at most $k-1$, it follows easily from Fact \ref{mds} that for any distinct $A,B\in\ma{H}_{\ma{M}}$, $|A\cap B|\le k-1$.

  By induction hypothesis, there exists a $\mr{G}_r(3r-2k,3)$-free $r$-graph $\ma{H}$ on $q$ vertices with at least $\frac{q^k}{r^k-r}-aq^{k-1+\delta}$ edges, and for any distinct $A,B\in\ma{H}$, $|A\cap B|\le k-1$.
  For each $1\le i\le r$, put a copy of $\ma{H}$, denoted as $\ma{H}_i$, into the vertex set $V_i$.
  Let $\ma{F}$ be the $r$-graph formed by the union $$\ma{F}:=(\cup_{i=1}^r\ma{H}_i)\cup\ma{H}_{\ma{M}}.$$
  It is not hard to see that for any $A,B\in\ma{F}$, $|A\cap B|\le k-1$.
  Moreover, since each of $\ma{H}_1,\ldots,\ma{H}_r,\ma{H}_{\ma{M}}$ is $\mr{G}_r(3r-2k,3)$-free, it is routine to check that $\ma{F}$ is also $\mr{G}_r(3r-2k,3)$-free. We omit the details here.

  It remains to prove an appropriate lower bound for $|\ma{F}|$.
  Clearly,
  \begin{equation*}
    \begin{aligned}
      |\ma{F}|&=|\ma{H}_{\ma{M}}|+r|\ma{H}|\ge q^k+r\cdot(\frac{q^k}{r^k-r}-aq^{k-1+\delta})\\
      &=\frac{(rq)^k}{r^k-r}-arq^{k-1+\delta}\ge\frac{(n-\frac{n^{\delta}}{r^{\delta-1}})^k}{r^k-r}-ar(\frac{n}{r})^{k-1+\delta}\\
      &\ge\frac{n^k}{r^k-r}-n^{k-1+\delta}(\fr{k}{r^{k+\delta-1}-r^{\delta}}+\fr{a}{r^{k+\delta-2}}).\\
    \end{aligned}
  \end{equation*}
  A short calculation shows that for large enough $a$, $\frac{k}{r^{k+\delta-1}-r^{\delta}}+\frac{a}{r^{k+\delta-2}}\le a$.
  Therefore, we conclude that $|\ma{F}|\ge \frac{n^k}{r^k-r}-an^{k-1+\delta}$ for some appropriate $a=a(r,k,n_0)$, completing the proof of the claim.
\end{proof}

\section*{Acknowledgements}

\noindent 
To the memory of Kobe Bryant, whose spirit inspires us.  
Thanks to Xin Wang and Xiangliang Kong for reading an early version of this manuscript, and for many helpful discussions. Thanks to an anonymous reviewer for his/her careful reading and constructive comments that helped us to improve this paper. 
Lastly, the research of C. Shangguan and I. Tamo was supported by ISF grant No. 1030/15 and NSF-BSF grant No. 2015814.

{\small
\bibliographystyle{plain}
\bibliography{turan_densities}

\begin{thebibliography}{10}

\bibitem{AlonShapira}
N.~Alon and A.~Shapira.
\newblock On an extremal hypergraph problem of {B}rown, {E}rd{\H o}s and
  {S}\'os.
\newblock {\em Combinatorica}, 26(6):627--645, 2006.

\bibitem{Alon1}
N.~Alon and A.~Shapira.
\newblock A characterization of the (natural) graph properties testable with
  one-sided error.
\newblock {\em SIAM J. Comput.}, 37(6):1703--1727, 2008.

\bibitem{alon2016probabilistic}
N.~Alon and J.~H. Spencer.
\newblock {\em The probabilistic method}.
\newblock John Wiley \& Sons, 2016.

\bibitem{Alon3}
N.~Alon and B.~Sudakov.
\newblock Disjoint systems.
\newblock {\em Random Structures Algorithms}, 6(1):13--20, 1995.

\bibitem{Alon2}
N.~Alon and R.~Yuster.
\newblock On a hypergraph matching problem.
\newblock {\em Graphs Combin.}, 21(4):377--384, 2005.

\bibitem{BakerPrimeNumber}
R.~C. Baker, G.~Harman, and J.~Pintz.
\newblock The difference between consecutive primes. {II}.
\newblock {\em Proc. London Math. Soc. (3)}, 83(3):532--562, 2001.

\bibitem{Behrend46}
F.~A. Behrend.
\newblock On sets of integers which contain no three terms in arithmetical
  progression.
\newblock {\em Proc. Nat. Acad. Sci. U. S. A.}, 32:331--332, 1946.

\bibitem{bohman2019large}
T.~Bohman and L.~Warnke.
\newblock Large girth approximate {S}teiner triple systems.
\newblock {\em J. London Math. Soc.}, 100(3):895--913, 2019.

\bibitem{BrownTriangulated}
W.~G. Brown, P.~Erd\H{o}s, and V.~T. S\'{o}s.
\newblock On the existence of triangulated spheres in {$3$}-graphs, and related
  problems.
\newblock {\em Period. Math. Hungar.}, 3(3-4):221--228, 1973.

\bibitem{BES71}
W.~G. Brown, P.~Erd{\H{o}}s, and V.~T. S{\'o}s.
\newblock Some extremal problems on {$r$}-graphs.
\newblock In {\em New directions in the theory of graphs ({P}roc. {T}hird {A}nn
  {A}rbor {C}onf., {U}niv. {M}ichigan, {A}nn {A}rbor, {M}ich, 1971)}, pages
  53--63. Academic Press, New York, 1973.

\bibitem{Bukhturanexponent}
B.~Bukh and D.~Conlon.
\newblock Rational exponents in extremal graph theory.
\newblock {\em J. Eur. Math. Soc. (JEMS)}, 20(7):1747--1757, 2018.

\bibitem{Erdos1964}
P.~Erd\H{o}s.
\newblock Extremal problems in graph theory.
\newblock In {\em Theory of {G}raphs and its {A}pplications ({P}roc. {S}ympos.
  {S}molenice, 1963)}, pages 29--36. Publ. House Czechoslovak Acad. Sci.,
  Prague, 1964.

\bibitem{Erdos-Turan}
P.~Erd\H{o}s.
\newblock On extremal problems of graphs and generalized graphs.
\newblock {\em Israel J. Math.}, 2:183--190, 1964.

\bibitem{Erdosr=2}
P.~Erd\H{o}s.
\newblock Problems and results in combinatorial analysis.
\newblock pages 3--17. Atti dei Convegni Lincei, No. 17, 1976.

\bibitem{ErdosTuranExponents1}
P.~Erd\H{o}s.
\newblock On the combinatorial problems which {I} would most like to see
  solved.
\newblock {\em Combinatorica}, 1(1):25--42, 1981.

\bibitem{ErdosTuranExponents2}
P.~Erd\H{o}s.
\newblock Problems and results in combinatorial analysis and graph theory.
\newblock In {\em Proceedings of the {F}irst {J}apan {C}onference on {G}raph
  {T}heory and {A}pplications ({H}akone, 1986)}, volume~72, pages 81--92, 1988.

\bibitem{EPR}
P.~Erd{\H{o}}s, P.~Frankl, and V.~R{\"o}dl.
\newblock The asymptotic number of graphs not containing a fixed subgraph and a
  problem for hypergraphs having no exponent.
\newblock {\em Graphs Combin.}, 2(2):113--121, 1986.

\bibitem{fitch2016rational}
M.~Fitch.
\newblock Rational exponents for hypergraph {T}ur\'an problems.
\newblock {\em J. Comb.}, 10(1):61--86, 2019.

\bibitem{Frankl1987}
P.~Frankl and Z.~F\"{u}redi.
\newblock Colored packing of sets.
\newblock In {\em Combinatorial design theory}, volume 149 of {\em
  North-Holland Math. Stud.}, pages 165--177. North-Holland, Amsterdam, 1987.

\bibitem{Fredman84}
M.~L. Fredman and J.~Koml\'{o}s.
\newblock On the size of separating systems and families of perfect hash
  functions.
\newblock {\em SIAM J. Algebraic Discrete Methods}, 5(1):61--68, 1984.

\bibitem{FurediCance2012}
Z.~F\"{u}redi.
\newblock 2-cancellative hypergraphs and codes.
\newblock {\em Combin. Probab. Comput.}, 21(1-2):159--177, 2012.

\bibitem{sparse}
Z.~F\"{u}redi and M.~Ruszink\'{o}.
\newblock Uniform hypergraphs containing no grids.
\newblock {\em Adv. Math.}, 240:302--324, 2013.

\bibitem{furedi2013history}
Z.~F{\"u}redi and M.~Simonovits.
\newblock The history of degenerate (bipartite) extremal graph problems.
\newblock In {\em Erd{\H{o}}s Centennial}, pages 169--264. Springer, 2013.

\bibitem{ge2017sparse}
G.~Ge and C.~Shangguan.
\newblock Sparse hypergraphs: new bounds and constructions.
\newblock {\em arXiv preprint arXiv:1706.03306}, 2017.

\bibitem{glock2018triple}
S.~Glock.
\newblock Triple systems with no three triples spanning at most five points.
\newblock {\em Bull. London Math. Soc.}, 51:230--236, 2019.

\bibitem{glock2016existence}
S.~Glock, D.~K{\"u}hn, A.~Lo, and D.~Osthus.
\newblock The existence of designs via iterative absorption.
\newblock {\em arXiv preprint arXiv:1611.06827}, 2016.

\bibitem{glock2018conjecture}
S.~Glock, D.~K{\"u}hn, A.~Lo, and D.~Osthus.
\newblock On a conjecture of {E}rd{\H{o}}s on locally sparse {S}teiner triple
  systems.
\newblock {\em arXiv preprint arXiv:1802.04227}, 2018.

\bibitem{kns}
G.~O.~H. Katona, T.~Nemetz, and M.~Simonovits.
\newblock On a graph-problem of {T}ur{\'a}n in the theory of graphs.
\newblock {\em Matematikai Lapok}, 15:228--238, 1964.

\bibitem{Keevashsurvey}
P.~Keevash.
\newblock Hypergraph {T}ur\'{a}n problems.
\newblock In {\em Surveys in combinatorics 2011}, volume 392 of {\em London
  Math. Soc. Lecture Note Ser.}, pages 83--139. Cambridge Univ. Press,
  Cambridge, 2011.

\bibitem{keevash2014existence}
P.~Keevash.
\newblock The existence of designs.
\newblock {\em arXiv preprint arXiv:1401.3665}, 2014.

\bibitem{Kovari-turanbound}
T.~K\"{o}vari, V.~T. S\'{o}s, and P.~Tur\'{a}n.
\newblock On a problem of {K}. {Z}arankiewicz.
\newblock {\em Colloquium Math.}, 3:50--57, 1954.

\bibitem{Rodl}
B.~Nagle, V.~R\"{o}dl, and M.~Schacht.
\newblock Extremal hypergraph problems and the regularity method.
\newblock In {\em Topics in discrete mathematics}, volume~26 of {\em Algorithms
  Combin.}, pages 247--278. Springer, Berlin, 2006.

\bibitem{Rodlpacking}
V.~R\"{o}dl.
\newblock On a packing and covering problem.
\newblock {\em European J. Combin.}, 6(1):69--78, 1985.

\bibitem{Ruzsa-Szemeredi}
I.~Z. Ruzsa and E.~Szemer{\'e}di.
\newblock Triple systems with no six points carrying three triangles.
\newblock In {\em Combinatorics ({P}roc. {F}ifth {H}ungarian {C}olloq.,
  {K}eszthely, 1976), {V}ol. {II}}, volume~18 of {\em Colloq. Math. Soc.
  J\'anos Bolyai}, pages 939--945. North-Holland, Amsterdam-New York, 1978.

\bibitem{other2}
G.~N. S{\'a}rk{\"o}zy and S.~Selkow.
\newblock An extension of the {R}uzsa-{S}zemer\'edi theorem.
\newblock {\em Combinatorica}, 25(1):77--84, 2005.

\bibitem{other1}
G.~N. S{\'a}rk{\"o}zy and S.~Selkow.
\newblock On a {T}ur\'an-type hypergraph problem of {B}rown, {E}rd{\H o}s and
  {T}. {S}\'os.
\newblock {\em Discrete Math.}, 297(1-3):190--195, 2005.

\bibitem{shangguanperfecthash}
C.~Shangguan and G.~Ge.
\newblock Separating hash families: a {J}ohnson-type bound and new
  constructions.
\newblock {\em SIAM J. Discrete Math.}, 30(4):2243--2264, 2016.

\bibitem{shangguan2019}
C.~Shangguan and I.~Tamo.
\newblock Universally sparse hypergraphs with applications to coding theory.
\newblock {\em arXiv preprint arXiv:1902.05903}, 2019.

\bibitem{Szemeredi}
E.~Szemer\'{e}di.
\newblock Regular partitions of graphs.
\newblock In {\em Probl\`emes combinatoires et th\'{e}orie des graphes
  ({C}olloq. {I}nternat. {CNRS}, {U}niv. {O}rsay, {O}rsay, 1976)}, volume 260
  of {\em Colloq. Internat. CNRS}, pages 399--401. CNRS, Paris, 1978.

\bibitem{turan}
P.~Tur\'{a}n.
\newblock Eine {E}xtremalaufgabe aus der {G}raphentheorie.
\newblock {\em Mat. Fiz. Lapok}, 48:436--452, 1941.

\end{thebibliography}
}
 \end{document}